\documentclass[a4paper,oneside]{amsart}
\usepackage{etoolbox}
\usepackage{amssymb,amsmath}
\usepackage{graphicx}
\usepackage{extarrows}
\usepackage{nicefrac}

\usepackage{thmtools}

\ifdef{\ebook}{
  \usepackage[paperwidth=9cm, paperheight=12cm, hmargin={0.17in, 0.17in}, vmargin={0.50in, 0.17in}]{geometry} \usepackage{kmath,kerkis}
}{}

\usepackage{longtable}

\usepackage{xypic}
\xyoption{curve}

\usepackage{tikz-cd}

\usepackage{booktabs}
\usepackage{color}
\usepackage{booktabs}
\usepackage{paralist}

\usepackage[draft,multiuser]{fixme}

\fxsetup{layout=inline}
\FXRegisterAuthor{p}{anp}{AP\!\!}

\FXRegisterAuthor{s}{ans}{HS} 
\FXRegisterAuthor{m}{anm}{MM} 

\declaretheorem[numberwithin=section,name=Theorem]{theorem}
\newtheorem{lemma}[theorem]{Lemma}
\newtheorem{proposition}[theorem]{Proposition}
\newtheorem{corollary}[theorem]{Corollary}
\theoremstyle{remark}
\newtheorem{remark}[theorem]{Remark}

\theoremstyle{definition}

\newtheorem{definition}[theorem]{Definition}
\newtheorem{example}[theorem]{Example}

\newtheorem{question}[theorem]{Question}

\newcommand{\CC}{\mathbb{C}}

\newcommand{\VV}{\mathbb{V}}
\newcommand{\QQ}{\mathbb{Q}}
\newcommand{\Z}{\mathbb{Z}}
\newcommand{\G}{\mathbb{G}}
\newcommand{\ZZ}{\mathbb{Z}}
\newcommand{\NN}{\mathbb{N}}
\newcommand{\PP}{\mathbb{P}}
\renewcommand{\P}{\mathbb{P}}
\newcommand{\A}{\mathbb{A}}

\newcommand{\K}{\mathbf{k}}
\newcommand{\KK}{\K}

\newcommand{\reg}{\mathrm{reg}}
\newcommand{\sing}{\mathrm{sing}}
\newcommand{\embed}{\hookrightarrow}
\def\C{\CC}
\def\Z{\ZZ}
\def\O{\mathcal{O}}
\def\P{\PP}
\newcommand{\fan}{\mathcal{S}}

\newcommand{\CO}{{\mathcal{O}}}

\renewcommand{\O}{{\mathcal{O}}}

\newcommand{\X}{{\mathfrak{X}}}
\newcommand{\sdeg}{\deg \fan}
\newcommand\0{\circ}

\DeclareMathOperator{\Ind}{Ind}

\DeclareMathOperator{\Sec}{Sec}
\DeclareMathOperator{\Tan}{Tan}

\DeclareMathOperator{\id}{id}
\DeclareMathOperator{\rep}{rep}
\DeclareMathOperator{\AffCone}{AffCone}
\DeclareMathOperator{\Cox}{\mathcal{R}}
\DeclareMathOperator{\SAut}{SAut}

\DeclareMathOperator{\lin}{lin}

\DeclareMathOperator{\tail}{tail}

\DeclareMathOperator{\im}{im}

\DeclareMathOperator{\Aut}{Aut}

\DeclareMathOperator{\supp}{supp}

\DeclareMathOperator{\conv}{conv}

\DeclareMathOperator{\Spec}{Spec}
\DeclareMathOperator{\Proj}{Proj}
\DeclareMathOperator{\Cl}{Cl}

\author[M. Micha\l ek]{Mateusz Micha\l ek}
\address{Polish Academy of Sciences, Mathematical Institut, \'{S}niadeckich 8, 00-956 Warszawa, Poland, Freie Universit\"at, Berlin, Germany}
\email{mmichalek@impan.pl}
\author[A. Perepechko]{Alexander Perepechko}
\thanks{The research of M.~Michalek was supported by  IP  grant
0301/IP3/2015/73 of the Polish Ministry of Science. The research of A.~Perepechko was carried out at the IITP RAS at the expense of the Russian 
Foundation for Sciences (project no. 14-50-00150).}
\address{Kharkevich Institute for Information Transmission Problems, 19 Bolshoy Karetny per., 127994 Moscow, Russia}  
\email{perepeal@gmail.com}
\author[H. S\"u{\ss}]{Hendrik S\"u\ss}
\address{
School of Mathematics,
The University of Manchester,
Alan Turing Building,
Oxford Road
Manchester M13 9PL}
\email{hendrik.suess@manchester.ac.uk}

\title[Flexible cones]{Flexible affine cones and flexible coverings}

\begin{document}
\begin{abstract}
  We provide a new criterion for flexibility of cones over varieties covered by flexible affine varieties. 
  We apply this criterion to prove flexibility of affine cones over secant varieties of Segre--Veronese embeddings and over certain Fano threefolds.
  We further prove flexibility of total coordinate spaces of Cox rings of del Pezzo surfaces.
\end{abstract}

\maketitle
Let $X$ be an affine variety over an algebraically closed field $\KK$ of characteristic zero.
We consider actions of the additive group $\G_a = (\KK,+)$ on $X$. 
The subgroup of $\Aut(X)$ generated by all $\G_a$-actions is called the \emph{special automorphism group} of $X$ and will be denoted by $\SAut(X)$. 
We are interested in transitivity of the $\SAut(X)$-actions on the smooth points of $X$. 
Recall that an action of a group $G$ on a set $M$ is called \emph{$m$-transitive} if for every two $m$-tuples $(x_1, \ldots, x_m)$ and $(x'_1, \ldots, x'_m)$ of pairwise distinct elements of $M$ 
there exists $g\in G$ such that $g.x_i = x_i'$ for $i=1, \ldots m$.

We have the following remarkable result.
\begin{theorem}[{\cite[Theorem~0.1]{AFKKZ}}]
\label{thm:AFKKZ}
Let $X$ be an irreducible affine variety of dimension $\ge 2$.
Then the following conditions are equivalent.
\begin{enumerate}
\item The group $\SAut(X)$ acts transitively on the regular locus $X_{\reg}$.
\item The group $\SAut(X)$ acts $m$-transitively on $X_{\reg}$ for every $m > 0$.
\item The tangent space of every $x \in X_{\reg}$ is spanned by tangent vectors to orbits of $\G_a$-actions.
\end{enumerate}
\end{theorem}
If these equivalent conditions are fulfilled the affine variety $X$ is called \emph{flexible}. 
As examples of flexible varieties, let us mention affine cones over del Pezzo surfaces of degree 4 and 5  \cite{flexcones}, 
 over flag varieties, and affine toric varieties without torus factors \cite{flex-instances}. 

Our main theorem provides a new criterion for flexibility of affine cones, see Section~\ref{sec:cone-flex} for the proof.
\begin{restatable*}{theorem}{flexcrit}
\label{thm:flexible-criterion}
Let $Y$ be a projective variety covered by flexible normal open affine subsets $U_i$ such that $Y\setminus U_i=D_i$, 
where $D_i$ are effective divisors equivalent to a very ample divisor $H$. Then the affine cone $X= \AffCone_H Y$ is flexible.
\end{restatable*}

Recall that a Segre--Veronese variety is an embedding of $\PP^{d_1}\times\dots\times\PP^{d_n}$ by the very ample line bundle $\O(s_1)\boxtimes\dots\boxtimes\O(s_n)$. Further,
 given a projective variety $X\subset\PP^n$, the Zariski closure of the union of the secant (resp. tangent) lines to $X$ is called a \emph{secant} (resp. \emph{tangential}) variety of $X$. 
 As the first application of Theorem~\ref{thm:flexible-criterion},
We deduce the following result, see Section~\ref{sec:secant} for details.
\begin{restatable*}{theorem}{secflex}
\label{thm:secant-verones}
Let $X=v_{s_1}(\P(V_1))\times\dots\times v_{s_n}(\P(V_n))$ be a Segre--Veronese variety.
Then the affine cone over the secant variety of $X$ is flexible.
Further, if $s_1=\dots =s_n=1$, then also the affine cone over the tangential variety of $X$ is flexible.
\end{restatable*}

Our proof technique relies on linearization of the affine charts of the ambient projective space by triangular transformations. They are inspired by algebraic statistics, precisely by computation of cumulants \cite{ciliberto2014cremona, MOZ, ZwiernikSturmfels, zwiernik, zwiernikbook}. 

Section~\ref{sec:torus-actions} contains preliminaries on smooth rational T-varieties of complexity $1$. 
These are varieties $X$ with an effective action of a torus $T$, where $\dim X = \dim T+1$. 
Section~\ref{sec:tvars} is devoted to flexibility of affine cones over such varieties. 

In \cite{APS13} it was shown that smooth varieties of this type admit a toric covering and for certain cones over these varieties we, indeed, obtain flexibility. 
For example, this applies to all known Fano threefolds with 2-torus action.
We use below the list of Fano threefolds in Mori--Mukai's classification \cite{Mori1981}. 
\begin{restatable*}{theorem}{conefanos}
\label{thm:cones-over-Fanos}
  All affine cones over the Fano threefolds $Q$, $2.29$, $2.30$,
  $2.31$, $2.32$, $3.8$, $3.18$, $3.19$, $3.23$, $3.24$, $4.4$,
  and elements of the families $2.24$, $3.10$ admitting a 2-torus action in Mori--Mukai's classification are flexible.
\end{restatable*}
The main tool to obtain these results is the combinatorial description of T-varieties developed in \cite{pre05013675,1159.14025}, which in the case of complexity $1$ allows to study (torus equivariant) coverings as in Theorem~\ref{thm:flexible-criterion}.

While Theorems \ref{thm:secant-verones} and \ref{thm:cones-over-Fanos} are concerned with projective coordinate rings, 
in Section~\ref{sec:total-coord-spaces} of the paper we obtain related results for total coordinate rings or Cox rings.
\begin{restatable*}{theorem}{coxdPflex}\label{thm:del-pezzo-cox-flexible}
  The total coordinate spaces of smooth del Pezzo surfaces are flexible.
\end{restatable*}

This was known so far only for the toric del Pezzo surfaces (i.e. those of degree 9,8,7 and 6) and by \cite[Thm.~0.2]{flex-instances} for the case of degree 5, 
where the total coordinate space is known to be the affine cone over the Grassmannian $G(2,5)$. 
On the other hand, this extends a result of \cite{APS13}, where flexibility was proved only outside a subset of codimension $2$.
\begin{restatable*}{theorem}{coxTflex}\label{T1-flex}
  The total coordinate space of a complete smooth T-variety of complexity one is flexible.
\end{restatable*}

\subsection*{Acknowledgments}
We would like to thank Mikhail Zaidenberg for motivating questions and inspiring results and Ivan Arzhantsev for many useful remarks and suggestions. 
The first author started the project under Mobilnosc+ Polish Ministry of Science program, finished under DAAD PRIME program and was supported by the Foundation for Polish Science (FNP).
The second author is grateful to the Dynasty Foundation for a partial financial support.

\section{Flexibility of affine cones}\label{sec:cone-flex}
\begin{lemma}\label{covering-orbit}
Let $X\subset\A^{n+1}$ be the affine cone over a projective variety
$Y\subset\PP^n$, smooth in codimension $1$, which is not isomorphic to an affine space.
Consider a subgroupa subgroup $G\subset\Aut X$ such that
\begin{itemize}
\item the canonical $\K^*$-action maps $G$-orbits to $G$-orbits and
\item there is an orbit $Gx\subset X$,
 whose image under the projection $X\setminus\{0\}\to Y$ is the regular locus $Y_{\reg}\subset Y$.
\end{itemize}
 Then $Gx=X_{\reg}$.
\end{lemma}
\begin{proof}
As the action is equivariant, $X_{\reg}$ is a union of $G$-orbits, which projections coincide with $Y_{\reg}$.  Hence $X_{\reg}=\bigcup_{\lambda\in\G_m}\lambda Gx$, where all $G$-orbits are closed in $X_\reg$.

 Let us show that there exists an open $G$-orbit $Gx=X_{\reg}$. Assume the contrary. Then $\dim Gx = \dim Y$ and the stabilizer $S\subset\G_m$ of the orbit $Gx$ is finite.

  Denote by $X^\times$   the blow up of $X$ at $0$. It is the line bundle $\mathcal{O}_Y(-1)$ over $Y$.
  Consider a quotient morphism $$\mu\colon X^\times\xlongrightarrow{/S} X'= X^\times/S$$ given by $t\mapsto t^{|S|}$ on every trivialization chart, where $t$ is the coordinate of the 
  fibers. Then $\mu(Gx)$ intersects each fiber at most once, so it is a meromorphic, nonvanishing section of the line bundle $X'\to Y$. 
  But the subset $D\subset Y$, where the meromorphic section of the line bundle is not defined or zero, if non-empty, is a Cartier divisor, since it is defined locally by rational functions. 
  On the other hand, $D\subseteq Y_{\sing}$ is of codimension at least $2$. 
  Therefore, $D$ is empty and $\mu(Gx)$ is a global section of $X'$, but $X'$  is not a trivial bundle, as $Y$ is not a point. A contradiction.
 
     So, the group $G$ acts on $X_{\reg}$ transitively.
\end{proof}

\begin{remark}
We can weaken the second condition. Namely, if there is a $G$-orbit on $X$, whose image has complement of codimension $\ge2$ in $Y$,
then this orbit is an open subset.
\end{remark}

\begin{lemma}\label{lifting}
Let $Y$ be an irreducible projective variety embedded into $\PP^n$, and $U=Y \setminus \{x_0 = 0\}$ be a complement to a hyperplane section endowed with a $\G_a$-action $\phi\colon \G_a\times U\to U$. 
Let $\pi\colon X\setminus\{x_0=0\}\to U$ be a natural projection, where $X=\AffCone Y\subset \A^{n+1}$ is the affine cone over $Y$. 
Then $X$ admits a homogeneous $\G_a$-action $\hat\phi\colon \G_a\times X\to X$ such that
\begin{itemize}
\item $\phi(\G_a\times\{\pi(x)\})=\pi(\hat\phi(\G_a\times\{x\}))$ for any $ x\in\pi^{-1}(U)$.
\item $\hat\phi$ is trivial on $X\setminus\pi^{-1}(U)$.
\end{itemize}
In other words, $\pi\colon X\setminus\{x_0=0\}\to U$ provides an infinity-to-one correspondence between $\hat\phi$-orbits and $\phi$-orbits.
\end{lemma}
\begin{proof}
Let $Y$ be defined by a homogeneous ideal $I\subset \K[x_0,\ldots,x_n]$ which does not contain $x_0$, $\K[X]=\K[x_0,\ldots,x_n]/I$, and $U=\{x_0\neq0\}\subset Y$.

There exists a natural embedding $\rho\colon U\embed X$, $\rho(U)=\{x_0=1\}\subset X$. 
On the other hand,  $\pi\colon X\setminus\{x_0=0\}\to U$ is a trivial $\G_m$-bundle. 
Therefore, we may extend the $\G_a$-action $\phi$ on $U$ to a $\G_a$-action $\tilde\phi$ on $X\setminus\{x_0=0\}$ defined by a homogeneous locally nilpotent derivation $\tilde\delta$.

There exists $d\in\NN$ such that $x_0^d\tilde\delta(x_i)\in\K[X]$ for $i=1,\ldots,n$. Since $x_0\in\ker\tilde\delta$, a homogeneous derivation $\hat\delta=x_0^{d+1}\tilde\delta$ on $X$ is locally nilpotent. 
The corresponding $\G_a$-action $\hat\phi$ is homogeneous, coincides with $\phi$ on $U\cong\{x_0=1\}\subset X$, and is trivial on $\{x_0=0\}\subset X$.
\end{proof}


\flexcrit
\begin{proof}
The case $X=\A^{n+1}$ is trivial. We assume in sequel that the origin is a
singular point of $X$.

For each element of the covering $U_i$ there exists a finite number of $\G_a$-actions $\{\phi_{ij}\}$ such that the orbit of the group generated by them is the regular locus of $U_i$.
For each action $\phi_{ij}$ we can consider a lifted action $\hat \phi_{ij}$ on $X$ as in Lemma~\ref{lifting}. 
Let a subgroup $\{\hat \phi_{ij}\}\subset\SAut X$ be generated by $\G_a$-actions corresponding to all open subsets $U_i$. 
Then the image of the orbit of a regular point $Gx\subset X_{\reg}$ under projection $X\setminus\{0\}\to Y$ is equal to $Y_{\reg}$. 
Thus, the statement follows from Lemma~\ref{covering-orbit}, as the variety $Y$ is normal, in particular smooth in codimension one.
\end{proof}




\begin{remark}
We note that in Theorem \ref{thm:flexible-criterion} it is enough to assume that there exists an effective $\QQ$-divisor $D'$ that has same support as $D$ and is equivalent to $H$. 
Indeed, let $k>0$ be such that $kD'$ is a ($\ZZ$-)divisor, then we obtain $\G_a$-actions on the affine cone over the Veronese embedding $v_k(Y)$. 
Following the proof of \cite[Corollary 3.2]{KPZ} one can lift those actions to the cone over $Y$ by \cite[Theorem 3.1]{MR2541410}.
\end{remark}

\begin{remark}
This criterion does not follow from \cite[Thm 6]{flexcones}. Indeed, the $\G_a$-actions on a flexible chart
correspond to cylinders that are not necessarily $H$-polar.
\end{remark}

\begin{example}
\label{exp:blowup}
Consider $\PP^n$ with coordinates $x_0, \ldots, x_n$ and a smooth subvariety $Y=\VV(x_0,f)$ of codimension $2$,
where $f$ is a homogeneous polynomial of degree $d$. 
Let $q \colon X \rightarrow \PP^n$ be the blowup in $Y$ and $q'\colon X' \rightarrow \PP^n$ the combined blowup in $Y$ and in the point $y=(1\mathbin:0:\cdots:0)$. 
We apply Theorem~\ref{thm:flexible-criterion} to show that all affine cones over $X$ and $X'$ are flexible. 
Note that for $d \leq n$ we obtain Fano varieties via this construction.

Let $U_i := \PP^n \setminus H_i := \PP^n \setminus \VV(x_i)$. 
The preimage $q^{-1}(U_0)$ is isomorphic to $U_0 \cong \A^n$, since $Y$ does not intersect $U_0$. 
For $i \neq 0$ the preimage $q^{-1}(U_i)$ is given by 
\[V\left(\frac{f}{x_i^d} \cdot u - \frac{x_0}{x_i} \cdot v\right) \quad \subset\quad  U_i \times \PP^1.\]
Hence, $q^{-1}(U_i)$ is covered by the affine charts 
\[U_i^0:=q^{-1}(U_i) \setminus [v=0] \text{ and } U_i^\infty:=q^{-1}(U_i) \setminus [u=0], \]
the first one being an affine space and the second one being isomorphic to 
\[V\left(\frac{f}{x_i^d} - \frac{x_0}{x_i} \cdot \frac{v}{u}\right) \subset \Spec k\left[\frac{x_0}{x_i},\ldots ,\frac{x_n}{x_i},\frac{v}{u}\right].\]
In the notation of \cite{flex-instances} this is a suspension over $\A^{n-1}$ and hence flexible by Theorem~0.2 loc. cit.

Given a divisor $D\subset\PP^n$, we denote by $\widetilde D$ its strict transform on the blowup.
We also denote by $E$ the exceptional divisor at $Y$ and by $H$ the pullback of the hyperplane.
To see that the given covering is polar with respect to every ample divisor, note that 
\begin{equation*}
  \label{eq:2}
  X \setminus U_0 = E \cup \widetilde H_0; \quad X \setminus U_i^\infty = q^*H_i \cup \widetilde H_0; \quad X \setminus U_i^0 = q^*H_i \cup \widetilde{\VV(f)}.
\end{equation*}

Now, $X$ has Picard group $\ZZ^2$ with generators $[H]$ and $[E]$. 
The effective cone is generated by $[E]$ and $[H]-[E]$ and the nef-cone by $[H]$ and $d[H]-[E]$.
Moreover, $[\widetilde{\VV(f)}] = d[H]-[E]$ and $[\widetilde H_0] = [H]-[E]$ hold.
So, for every of our affine charts we see that the components of the complement (\ref{eq:2}) span 
 a cone in the Neron-Severi space containing the whole nef cone of $X$. 
 Hence, every ample class can be expressed as a positive linear combination of the complement components.

Similarly, for $X'$ we have the flexible charts 
\[U^i_0 := X \setminus (\widetilde H_0 \cup E \cup \widetilde H_i),\quad 
 U_i^0:=X \setminus (\widetilde H_i \cup \widetilde{\VV(f)} \cup E'),\quad 
U_i^\infty:=X \setminus (\widetilde H_i \cup \widetilde{H}_0 \cup E'),\]
 the first two being affine spaces and the last one being a suspension as before. 
 Here, we use the same notation as above and $E'$  denotes the exceptional divisor of the blowup in the point. 
 We see that the complements of the affine charts always consist of three components with classes corresponding to one of the triples 
\[([H]-[E], [E], [H]-[E']),\quad ([H]-[E'], d[H]-[E], [E']),\quad([H]-[E'], [H]-[E], [E']).\] 
Further, each triple spans a cone containing the nef cone of $X'$, which is spanned by $d[H]-[E]$, $[H]-[E']$, and $[H]$. 
 Hence, every ample class can be expressed as a positive linear combination of complement components.
\end{example}

\section{Secant of Segre--Veronese variety}
\label{sec:secant}
This section is based on \cite{MOZ}, where the toric covering of a Segre variety was constructed.
Here we generalize that construction to a Segre--Veronese variety and hence prove Theorem~\ref{thm:secant-verones}.
Throughout this section by a \emph{parameterization} of a variety $Z$ 
we mean a dominant morphism from an open subset of an affine space to $Z$. 
\begin{definition}
Given for each $i=1,\ldots,n$ a finite-dimensional vector space $V_i$ and its symmetric power 
$S^{s_i}(V_i)$, $s_i\in\ZZ_{>0}$, the \emph{Segre--Veronese variety} 
\[X=v_{s_1}(\P(V_1))\times\dots\times v_{s_n}(\P(V_n))\subset\P(S^{s_1}(V_1)\otimes\dots\otimes S^{s_n}(V_n))\]
 is defined as the embedding of the product 
$\P(V_1)\times\dots\times\P(V_n)$ by the very ample line bundle $\O(s_1)\boxtimes\dots\boxtimes\O(s_n)$.
\end{definition}



We will be using an equivalent construction.
Apart from (projective) Segre--Veronese varieties we will consider affine cones over them and refer to those as \emph{Segre--Veronese cones}. 
They should not be confused with intersections of Segre--Veronese varieties with principal affine open subsets, which also play a crucial role. 

 For each $V_i$, $1\leq i\leq n$, we denote $d_i=\dim V_i-1$ and fix a basis $e^i_0,\dots,e^i_{d_i}$ of $V_i$. 
We also denote elements of the basis of $V_i^{\otimes s}$ by \[e^i_{i_1,\ldots,i_s}=e_{i_1}^i\otimes\dots\otimes e_{i_s}^i.\]
Thus, the symmetric power is $S^{s_i}(V_i)=$
\[\{v\in V_i^{\otimes s}\mid (e^i_{i_1,\ldots,i_s})^*(v)=(e^i_{i_{\sigma(1)},\ldots,i_{\sigma(s)}})^*(v)\text{ for any permutation }\sigma\in S_{s}\}.\]
This allows us to embed the Veronese cone into the Segre cone and obtain the following diagram:
\begin{center}
\includegraphics{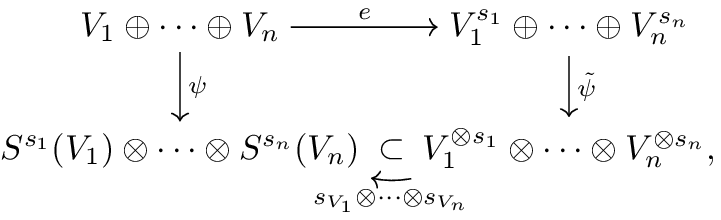}
\end{center}

where
\begin{itemize}
\item $e\colon(v_1,\dots,v_n)\mapsto (v_1,\dots,v_1,v_2,\dots,v_2,v_3,\dots,v_n)$ 
is the diagonal embedding,
\item  $\tilde \psi\colon (v_1^1,\dots,v_1^{s_1},v_2^1,\dots,v_2^{s_2},v_3^1,\dots,v_n^{s_n})\mapsto 
v_1^1\otimes\dots\otimes v_n^{s_n}$ 
is the parameterization the Segre cone, which is a nonlinear map,
\item  $\psi=\tilde\psi|_{\im(e)}\circ e$ 
is the parameterization of the Segre--Veronese cone, and
\item $s_{V_i}\colon V_i^{\otimes s_i}\rightarrow S^{s_i}(V_i)$ are the natural symmetrizing projections.
\end{itemize}

\begin{definition}[ $\hat{}$ , $A$, $B$]
\label{def:hat}\
\begin{enumerate}
\item For a vector space $V=V_{i_1}\otimes\cdots\otimes V_{i_k}$
we denote
\[\hat{V}:=\{x\in V\mid (e_0^{i_1}\otimes\cdots\otimes e_0^{i_k})^*(x)=1\}\]
and regard it as a vector space with basis 
$\{e_{j_1}^{i_1}\otimes\cdots\otimes e_{j_k}^{i_k}\mid j_1+\cdots+j_n>0\}.$ 
\item We may, and often will, consider $\hat{V}$ as a complement to a hyperplane section of $\P(V)$ by the natural open embedding $\hat{V}\subset \P(V)$.
\item We denote $A:=V_1^{\otimes s_1}\otimes\dots\otimes V_n^{\otimes s_n}$.
\item We also define
\[B:= \prod_{i=1}^n\hat{V_i}^{\times s_i}\subset V_1^{s_1}\oplus\dots\oplus V_n^{s_n},\]
\[B':= \prod_{i=1}^n\hat{V_i}\subset V_1\oplus\dots\oplus V_n.\]
\item
Finally, we denote $\pi\colon A\setminus\{0\}\to\P(A)$ and obtain the following diagram of open subsets:
\begin{center}
\includegraphics{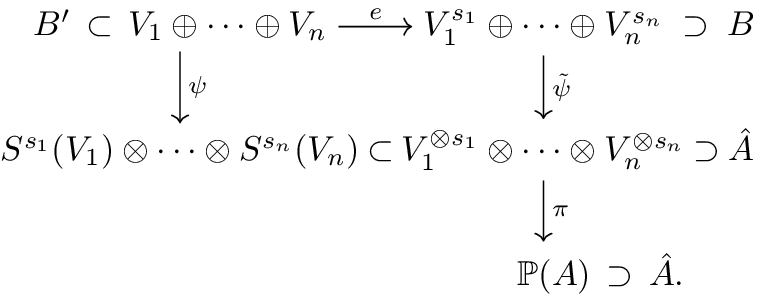}
\end{center}
\end{enumerate}
\end{definition}

\begin{remark}
Since $\P(S^{s_1}(V_1)\otimes\dots\otimes S^{s_n}(V_n))\subset\P(A)$, 
we can study the Segre--Veronese variety as a subvariety of $X=\overline{\pi\circ\psi(B')}\subset\P(A)$. 
Note that the image of $\tilde\psi|_B$ does not contain the origin.
\end{remark}

\subsection*{Cumulants}
In this setting we may apply the (nonlinear) coordinate systems of $B$, called cumulants and  presented in \cite{MOZ}. 
For the motivations to consider them, coming from algebraic statistics, we refer the reader to \cite{ZwiernikSturmfels, zwiernik, zwiernikbook}. 
A general mathematical setting for these methods is well presented in \cite{ciliberto2014cremona}. 
Further results are obtained for other varieties, e.g.~Grassmannians and spinor varieties \cite{manivel2014secants}. 
However, in other cases we do not obtain toric coverings. Still, we believe that similar methods can be applied to a larger class of secant and tangential varieties.

\begin{definition}[indices]\label{def:indices}
Basis elements of $A$ are of form $e^1_{c^1_1,\ldots,c^1_{s_1}}\otimes\cdots\otimes e^n_{c^n_1,\ldots,c^n_{s_n}}$ and are in natural correspondence with tuples
$(c^1_1,\ldots,c^1_{s_1},\ldots,c^n_1,\ldots,c^n_{s_n})$, 
where $ 0\le c^i_j\le d_i$ for $1\le i\le n$ and $1\le j\le s_i$. 
Let us denote the set of these tuples by $C(A)$ and for each  $c\in C(A)$  the corresponding basis element  by $e(c)$. 
Finally, denote the dual basis elements by $x(c)=e(c)^*$. Similarly, we denote $C(\hat A)=C(A)\setminus\{(0,\ldots,0)\}$.
\end{definition}

\begin{definition}[degree, ordering]\label{def:degree}
 Given a tuple $c\in C(\hat A)$, the number of its nonzero entries is called the \emph{degree} of $c$.
Given $c_1,c_2\in C(\hat A)$, we say that $c_1\le c_2$ if $c_1$ can be obtained from $c_2$ by setting some entries to zero.
\end{definition}
Thus, we have a natural poset structure on $C(\hat A)$, which induces a poset structure on the basis of
$\hat{A}$. 
Simply speaking, the ordering on the basis is defined by replacing  $e^i_j$ by $e^i_0$ in the tensor product elements.
\begin{example}
Consider $t_2:=e^1_1\otimes e^1_0\otimes e^2_5$, $t_1:=e^1_0\otimes e^1_0\otimes e^2_5$ and $t_0:=e^1_1\otimes e^1_3\otimes e^2_4$. 
We have $t_1<t_2$ and $t_1,t_2$ are not comparable with $t_0$.
\end{example}
\begin{definition}
Denote the set of indices
\[\textrm{Ind}(\hat A)=\textstyle\left\{{1\brack 1},\ldots,{1\brack s_1},\ldots,{n\brack 1},\ldots,{n\brack s_n}\right\}\]
and endow it with a natural lexicographic order, namely, the order of appearance above.
Given $c\in C(\hat A)$ and a subset of indices 
$I\subset\Ind(\hat A)$,
we introduce the index tuple 
\[c_I=(b^1_1,\ldots,b^n_{s_n}),\quad\mbox{ where } b^i_j=
\left[\begin{array}{lr}
c^i_j, & {i\brack j}\in I,\\
0,& {i\brack j}\notin I.
\end{array}\right.\]
Note that $\{b\in C(\hat A)\mid b\le c\}=\{c_I\mid I\subset \Ind(\hat A)\}$.
We will use either one-element subsets $I={i\brack j}$ or subsets of the following form, where  $i_1,i_2\in \Ind(\hat A)$:
\[
I=[i_1:i_2]:= \{i\mid i_1\le i< i_2\}\subset\Ind(\hat A).
\]
\end{definition}

\begin{definition}
A \emph{thick interval partition} of a tuple $c\in C(\hat A)$ of degree at least two is an increasing sequence of indices ${1\brack 1}=b_0<\ldots<b_k={n\brack s_n}$ 
such that $\deg c_{[b_i:b_{i+1}]}\ge2$ for each $i$. 
The set of all thick interval partitions of $c$ will be denoted by $IP(c)$. 
It is always nonempty 
as it contains $\left\{{1\brack 1},{n\brack s_n}\right\}$.
\end{definition}

Now we can recall the coordinate systems.

\begin{definition}[{\cite[Sec. 2]{MOZ}}]\label{def:aut1}
For each $c\in C(\hat A)$ we denote
\[y(c):=
\left[\begin{array}{lr}
x(c), & \deg c=1,\\
\sum_{(b)\leq (c)}(-1)^{\deg(c)-\deg(b)}x(b)\prod_{c^{i_0}_{j_0}\neq b^{i_0}_{j_0}}x(c_{i_0\brack j_0}),&\deg c>1.
\end{array}\right.
\]
Then for each $c\in C(\hat A)$ we introduce a function in $\K[\hat A]$
\[z(c):=
\left[\begin{array}{lr}
y(c), & \deg c=1,\\
\sum_{(b_0,\dots,b_k)\in IP(c)}(-1)^{k}\prod_{m=1}^k y(c_{[b_{m-1}:b_m]}),&\deg c>1.
\end{array}\right.
\]
\end{definition}

\begin{lemma}
Each of sets $\{x(c)\}_{ c\in C(\hat A)},$ $\{y(c)\}_{ c\in C(\hat A)},$ and $\{z(c)\}_{ c\in C(\hat A)}$
is an algebraically independent system of functions generating $\O(\hat A)$. In other words, $\{z(c)\}$ is a coordinate system on $\hat A$ as an affine space.
\end{lemma}
\begin{proof}
Since $y(c)$ is a sum of $x(c)$ and of terms of smaller degree, the endomorphism of $\K[\hat A]$ that maps $x(c)$ to $y(c)$ for each $c$ is invertible. Same holds for $\{y(c)\}$ and $\{z(c)\}$. So, the statement follows.
\end{proof} 

\subsection*{Secant}
The secant variety $\Sec X\subset \P(A)$ of the Segre--Veronese variety $X$ is parameterized by a map
\begin{equation}\label{eq:secx}
\sec_X\colon \A^1\times B'\times B'\to \hat A, \quad (t,v,w)\mapsto \pi(t\cdot\psi(v)+(1-t)\cdot\psi(w)).
\end{equation}

Hereinafter, given a tuple of degree one, we denote the index of its only non-zero entry  by ${i_\0 \brack j_\0}$.
Generalizing \cite[Lemma~3.1]{MOZ}, we obtain the following result.
\begin{lemma}\label{lem:paramSV}
Let $\{(e_j^i)^{\prime*}\mid 1\le i\le n,\, 0\le j \le d_i\}$ be the set of coordinate functions on the first copy of $B'$ in (\ref{eq:secx})
and $\{(e_j^i)^{\prime\prime*}\}$ the respective set on the second one.\footnote{That is, $(e_j^i)^{\prime*}(v)=v_j^i$ and $(e_j^i)^{\prime\prime*}(w)=w_j^i$ respectively.}
Then 
\[\def\arraystretch{1.5}
\sec_X^*\colon z(c)\mapsto
\left[\begin{array}{l}
t(e^{i_\0}_{c^{i_\0}_{j_\0}})^{\prime*}+(1-t)(e^{i_\0}_{c^{i_\0}_{j_\0}})^{\prime\prime*}, 
\hfill \deg c=1\mbox{ with } c^{i_\0}_{j_\0}\neq0,\\
 t(1-t)(1-2t)^{\deg(z(c))-2}\prod_{c^i_j\neq 0}((e^{i}_{c^{i}_{j}})^{\prime*}-(e^{i}_{c^{i}_{j}})^{\prime\prime*}),
\hspace{2em}\hfill \deg c>1,
\end{array}\right.
\]
for each $c\in C(\hat A).$
\end{lemma}
\begin{proof}
Let $Y$ be the cone over the Segre product $\P(V_1)^{s_1}\times\dots\times\P(V_n)^{s_n}$.
Then the secant of $Y$ is parameterized by
\[\sec_Y:\A^1\times B\times B\to \hat{A},\; (t,v,w)\mapsto \pi( t\tilde\psi(v)+(1-t)\tilde\psi(w)).\]
Thus, $\sec_X=\sec_Y\circ(\id\times e\times e)$.
The statement follows after applying \cite[Lemma~3.1]{MOZ} to $\sec_Y$.
\end{proof}

\subsection*{Torus action}
We can infer the following decomposition of $\sec_X$.
\begin{definition}
Let us introduce a reparameterization $\rep\colon\A^1\times B'\times B'\to\A^1\times B'\times B',$
\[
\rep\colon(t,v,w)\mapsto\left(\frac{t(1-t)}{(1-2t)^2},tv+(1-t)w,(1-2t)(w-v)\right)
\]
and a \emph{monomial} map $m\colon\A^1\times B'\times B' \to\hat A,$
\[\def\arraystretch{1.5}
m^*\colon z(c)\mapsto
\left[\begin{array}{l}
(e^{i_\0}_{c^{i_\0}_{j_\0}})^{\prime*}, 
\hspace{2em}\hfill \deg c=1\mbox{ with } c^{i_\0}_{j_\0}\neq0,\\
t\prod_{c^i_j\neq 0} (e^i_{c^i_j})^{\prime\prime*},
\hspace{2em}\hfill \deg c>1.
\end{array}\right.
\]
\end{definition}
\begin{lemma}
There is a decomposition $\sec_X = m\circ\rep.$ In particular, $m$ is a monomial parameterization of $\Sec X$. 
\end{lemma}
\begin{proof}
Straightforward.
\end{proof}
This monomial parameterization of $\Sec X$ already provides a structure of a toric variety on $\im(m)=\hat A\cap\Sec X,$ hence provides us with a toric chart of $\Sec X$. Below we describe in detail its structure.

\begin{definition}
Let us introduce the following closed subsets of $\hat A$:
\begin{align*}
\hat A_1&= \{z(c)=0\mid \deg c>1\},\\
\hat A_2 &= \{z(c)=0\mid \deg c=1\},\\
S_2&=\Sec X\cap A_2.
\end{align*}
\end{definition}

\begin{definition}[$P$]
Consider the lattice $M=\bigoplus_{\substack{1\leq i\leq n\\ 1\leq j\leq d_i}}\Z \chi^i_j$, 
where $\chi^i_j=(e^i_j)^*$. We introduce the lattice polytope $P\subset M\otimes\QQ$ defined by inequalities
\[
\begin{cases}
(\chi^i_j)^*\geq 0,  &1\leq i\leq n, 1\leq j\leq d_i,\\
\sum_{j=1}^{d_i}(\chi^i_j)^*\leq s_i, & 1\leq i\leq n,\\
\sum_{\substack{1\leq i\leq n\\ 1\leq j\leq d_i}}(\chi^i_j)^*\geq 2.&
\end{cases}
\]
\end{definition}

\begin{proposition}\label{pr:sec}
In the terminology above, 
\begin{enumerate}
\item $\Sec X\cap \hat A= \hat A_1\times S_2$ via natural projections along coordinates $z(c)$,
\item $\hat A_1=X\cap \hat A\cong \A^N$, where $N={\sum_{i=1}^n\dim \hat V_i},$
\item $S_2\cong\AffCone(X_P)$, where $X_P$ is a projective toric variety with polarization corresponding to the polytope $P$, e.g. see \cite[\S2.3]{cox2011toric} for the construction.
\end{enumerate}
\end{proposition}
\begin{proof}
The morphism $m$ is a direct product of $m|_{\{0\}\times B'\times\{0\}}$ and $m|_{\A^1\times \{0\}\times B'}$, which respectively parameterize $\hat A_1$ and $S_2$. This implies (i).
The Segre--Veronese variety $X$ is also parameterized by $m|_{\{0\}\times B'\times \{0\}}$, thus (ii) holds.

Let us consider the standard torus $T=\Spec\K[M]\subset B'$ and a projectivization $\pi_z\colon \hat A\setminus\{0\}\to\P(\hat A)$ along $z(c)$-coordinates.
Then $X_P=\overline{\pi_z(S_2)}$ is parameterized by a $T$-equivariant map $\pi_z\circ m|_{\{0\}\times\{0\}\times B'}$, which is defined by the set of monomials $\{(e^i_{c^i_j})^{*}\mid\deg c>1\}\subset\O(B')$ corresponding exactly to lattice points $P\cap M$. This implies (iii).
\end{proof}


\begin{proposition}\label{pr:tan}
Let $\Tan X\subset \P(A)$ be the tangential variety of $X$. 
Then 
\[\Tan X\cap\hat A=\hat A_1\times X_P^\prime,\]
where $X_P^\prime\subset \hat A_2$ is a nondegenerate toric variety parameterized by  $m|_{\{-\frac{1}{4}\}\times \{0\}\times B'}$.
\end{proposition}
\begin{proof}
We present here a sketch of a proof for $\K=\C$, see \cite{MOZ} and  \cite[Lem.~3.3]{manivel2014secants} for a complete proof.
Consider $(\epsilon^{-1},v, v+\epsilon w)\in \A^1\times B'\times B'$.
If  $\epsilon\to0$, then $\sec_X(\epsilon^{-1},v, v+\epsilon w)$ tends to an element of  $\mathrm{T}_{\pi\circ\psi(v)}X\embed\Tan X$ corresponding to $w$.
On the other hand, by Lemma~\ref{lem:paramSV}, 
\[\def\arraystretch{1.5}
\lim_{\epsilon\rightarrow 0}z(c)(\sec_X(\epsilon^{-1},v, v+\epsilon w))\to
\left[\begin{array}{l}
v^{i_\0}_{c^{i_\0}_{j_\0}}-w^{i_\0}_{c^{i_\0}_{j_\0}}, 
\hspace{2em}\hfill \deg c=1\mbox{ with } c^{i_\0}_{j_\0}\neq0,\\
-\frac{1}{4}\prod_{c^i_j\neq 0} 2w^i_{c^i_j},
\hspace{2em}\hfill \deg c>1.
\end{array}\right.
\]
Thus, the decomposition follows. It remains to check that $X^\prime_P$ is nondegenerate. Indeed, $\O(X^\prime_P)\embed\O(B')$ does not contain invertible elements.
\end{proof}

These propositions imply the following relationship of the tangential and secant varieties,
which, in its turn, implies Zak's theorem \cite{Zak} for Segre--Veronese varieties.
\begin{corollary}\label{c:sectan}
The following conditions are equivalent:
\begin{enumerate}
\item $P$ is not contained in the hyperplane $\sum_{i,j}(\chi^i_j)^*= 2$,
\item $\dim\Sec X=2\dim X+1$,
\item $\Sec X\neq \Tan X,$
\item $\dim\Tan X = \dim \Sec X -1$.
\end{enumerate}
Then $\Sec X$ is called non-degenerate.
\end{corollary}
\begin{proof}

Let $d=\dim X$. As a toric variety, $X$ is represented by a $d$-dimensional polytope $S$, which is a product of simplices. Then $P$ is the intersection of $S$ with the halfspace $\sum_{i,j}(\chi^i_j)^*\geq 2$. 

$i)\Rightarrow ii)$: The assumption implies $\dim P=d$. Hence, $\dim S_2=d+1$ and $\dim\Sec X=\dim \hat A_1+\dim S_2=2d+1$.

$i)\Rightarrow iv)$: As before, $\dim \mathrm{Cone}(P)=d$. Hence, $\dim \Tan X=\dim\hat A_1+d=2d$.

The implications $ii)\Rightarrow iii)$ and $iv)\Rightarrow iii)$ are obvious.

$iii)\Rightarrow i)$: If $P$ is contained in the hyperplane, then all monomials corresponding to lattice points in $P$ are of same degree. In particular, $X_P'\cong \AffCone(X_P)$.
\end{proof}

\begin{example}
{\bf 1)} Consider the Veronese surface $Y=\PP^2\embed\PP^5$. It is a projective toric variety corresponding to the simplex $S=\conv(0,2\chi_1,2\chi_2)\subset\langle \chi_1,\chi_2\rangle\cong\ZZ^2$, i.e. parameterized by characters of a two-dimensional torus that correspond to the lattice points of $S$. 
By Proposition \ref{pr:sec}, the variety $X_P$ that defines the secant is parameterized by $P=S\cap\{\chi_1^*+\chi_2^*\geq 2\}$, i.e., by lattice points $2\chi_1,\chi_1+\chi_2,2\chi_2$. 
So, both factors of $\Sec X\cap\hat A=\hat A_1\times S_2$ are of dimension two, so $\dim\Sec X=4$.
Thus, the secant variety is degenerate and defined by $z(1,2)^2=z(1,1)z(2,2)$.
  
{\bf 2)} Consider the Segre product $\PP^1\times\PP^1\times\PP^1$. 
The representing polytope is a cube $[0,1]^3$. 
The secant variety is represented by a polytope with vertices $(1,1,0)$, $(1,0,1)$, $(0,1,1)$, $(1,1,1)$. 
We see that the affine cone over it is the whole affine space; indeed in this case the secant variety is non-degenerate and fills the whole ambient space. 
The tangential variety is a hypersurface defined by the equation
$z(1,1,0)z(1,0,1)z(0,1,1)=z(1,1,1)^2$.
\end{example}

\begin{theorem}
\label{thm:secant-veronese-covering}
The tangential and secant varieties of a Segre--Veronese variety are covered by complements of hyperplane section. Each complement is an affine toric variety without torus factors. In case of the secant variety, these are always normal toric varieties. In case of the tangential variety they are normal if the underlying variety is the Segre product.
\end{theorem}
\begin{proof}
By \cite[Theorem 2.2]{vermeire} we know that $\Sec X$ is normal. By \cite[Proposition 8.5]{MOZ} we know that $\Tan X$ is normal, when $X$ is the Segre product.
 The open subsets
 $\Sec X\cap \hat A$ and $\Tan X\cap \hat A$  are toric varieties by Propositions~\ref{pr:sec} and \ref{pr:tan}.
Moreover, they do not contain torus factors, since they are products of an affine space $\hat A_1$ with either an affine cone over a projective toric variety $X_P$ or a non-degenerate toric variety $X_P^\prime$. 
  By taking such subsets for various choices of basis vectors $e^i_0$, $i=1,\ldots,n$, we obtain the statement.
\end{proof}

\secflex
\begin{proof}
  By \cite[Theorem 0.2]{flex-instances} we know that all the affine charts from Theorem~\ref{thm:secant-veronese-covering} are flexible. Now, by Theorem~\ref{thm:flexible-criterion} we obtain flexibility of the cone.
\end{proof}
\begin{example}
Consider the third Veronese embedding $X=v_3(\PP^1)$, it is represented by characters in the interval $S=\conv(0,3)$. Then $\Tan X\cap \hat A\cong \A^1\times X_P'$, where the monoid of characters associated to the toric variety $X_P'$ is generated by $\{2,3\}$. Namely, $X_P'$ is the curve with a cusp singularity at the origin. Thus, $\Tan X$ is  a two-dimensional surface, whose singular locus is the curve $X$. This example can be generalized to tangential varieties of other Segre--Veronese varieties provided that at least one of the Veronese factors is of degree at least $3$. 
\end{example}

\section{The combinatorial description of T-varieties}
\label{sec:torus-actions}
We consider a normal variety $X$ with an effective action of an algebraic torus $T \cong (\G_m)^r$. 
Then $X$ is called a $T$-variety of complexity $(\dim X - \dim T)$. 
In the following we restrict ourselves to the case of rational $T$-varieties of complexity one. 
 Our general reference is \cite{tvars}.

Let us denote the character lattice of the torus $T$ by $M$ and the dual lattice by $N$. For the associated vector spaces we write $M_\QQ$ and $N_\QQ$.

For a polyhedron $\Delta \subset N_\QQ$ we consider its tail cone $\tail(\Delta) := \{v \in N_\QQ \mid \Delta + \QQ_{\geq 0} \cdot v = \Delta \}$. 
Now, we consider polyhedral complexes $\Xi$ in $N_\QQ$ such that the set of tail cones has the structure of a fan, which is called the tail fan of $\Xi$ and will be denoted by $\tail(\Xi)$. 
Consider a pair $\fan = (\sum_{P\in \PP^1} \fan_P \otimes P,\sdeg)$ where $\fan_P$ are polyhedral complexes in $N_\QQ$ with some common tail fan $\Sigma$ and $\sdeg \subset |\Sigma|$. 
Here,  $\sum_P \fan_P \otimes P$ is just a formal sum. 
The complexes  $\fan_P$ are called \emph{slices} of $\fan$. 
We assume that there are only finitely many slices that differ from the tail fan $\tail(\fan) := \Sigma$. 
The set of the points $P \in \PP^1$ such that $\fan_P \neq \Sigma$ is called the support of $\fan$ and will be denoted by $\supp \fan$. 
Note that for every full-dimensional $\sigma\in \Sigma$ there is a unique polyhedron $\Delta_P^\sigma$ in $\fan_P$ with $\tail(\Delta_P^\sigma)=\sigma$.
\begin{definition}[f-divisor]
  A pair $\fan$ as above is called an \emph{f-divisor} if for any full-dimensional $\sigma\in\tail(\fan)$ we have either $\sdeg\cap\sigma = \emptyset$ or \[\sum_P \Delta^\sigma_P = \sdeg\cap\sigma \subsetneq \sigma.\]
\end{definition}

An f-divisor as above corresponds to a rational $T$-variety of complexity one, see \cite[Section~1]{IS10}. 
Moreover, this correspondence is even functorial. 
In particular, invariant open subvarieties correspond to f-divisors $\fan'$, 
such that $\fan^{\prime}_P \subset \fan_P$ as sets of polyhedra and $\sdeg' = |\tail(\fan')| \cap \sdeg$. 
For simplicity we write $\fan' \subset \fan$ in this situation. 
As a consequence of Proposition~1.6 in \cite{IS10}, f-divisors $\fan_1, \ldots, \fan_\ell \subset \fan$ give rise to an open covering if and only if their slices cover the slices of $\fan$, i.e.
\[|\fan| = \bigcup_i |\fan_i|.\]
\begin{remark}
  Affine charts correspond to f-divisors $\fan$ such that $\fan_P$ consists of a single polyhedron (and its faces) and $\sdeg = \sum_{P\in \PP^1} \fan_P$. These objects are called p-divisors in \cite{IS10,tvars}.
\end{remark}

\begin{example}
\label{exp:fdiv}
In Figure~\ref{fig:fdiv} we sketched the non-trivial slices of an f-divisor as well as its degree. It describes the blowup of the quadric threefold in one point, see \cite{suess:picbook}.
  \begin{figure}[htbp]
    \centering
    \includegraphics{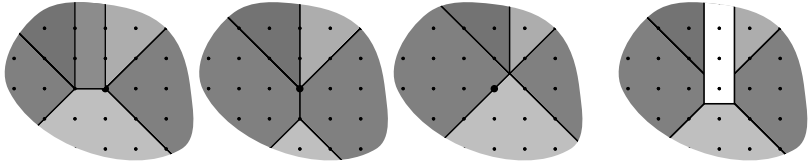}
    \caption{An f-divisor}
    \label{fig:fdiv}
  \end{figure}
\end{example}

\begin{lemma}[{\cite[Remark 1.8.]{1244.14044}}]
\label{lem:toric}
  An f-divisor describes a subtorus action on a toric variety if and only if
  $\fan_P$ equals a lattice translate of the tail fan for all but at most two $P \in \PP^1$.
\end{lemma}
In the language of f-divisors we also may describe torus invariant Cartier divisors by \emph{support functions}. 
A support function $h$ on a polyhedral subdivision $\Xi$ is a continuous function that is affine linear on every polyhedra in $\Xi$. 
We denote by $\lin h$ the linear part of $h$. 
This is a piecewise linear function on the tail fan defined as follows:
\[(\lin h)(v):=h(w+v)-h(w)\]
for some $w \in \Delta \in \Xi$ with $v \in \tail(\Delta)$.
\begin{definition}[Support function on $\fan$]
  A \emph{support function} $h$ on an f-divisor $\fan$ is a collection $\{h_P\}_{P \in \PP^1}$ of support functions on $\fan_P$ such that
  \begin{enumerate}
  \item all $h_P$  have the same linear part, which will be denoted by  $\lin h$,
  \item only finitely many of them differ from $\lin h$.
  \end{enumerate}
\end{definition}

We have two kinds of torus invariant prime divisors on $X(\fan)$. 
\emph{Horizontal} prime divisors correspond to rays $\rho \in \tail(\fan)^{(1)}$ that do not intersect $\sdeg$ and are denoted by $D_\rho$. 
\emph{Vertical} prime divisors correspond to vertices $v$ in the subdivisions $\fan_P$ and are denoted by $D_{P,v}$. 
Now, the divisor corresponding to the support function $h$ is given by
\[D_h = -\sum_\rho (\lin h)(\rho) \cdot D_\rho - \sum_{P,v} \mu(v) \cdot h_P(v) D_{P,v},\]
where we identify the ray with the ray generator and $\mu(v)$ denotes the minimal positive integer such that $\mu(v) \cdot v$ is a lattice element. In particular,
\begin{enumerate}
\item if $h_P\leq 0$ for all $P\in\PP^1$, then $D_h$ is effective.
\item in this case $X(\fan) \setminus\supp D_h$ is given by the f-divisor
\[[h=0]:= \left(\sum_P [h_P=0] \otimes P,\quad \sdeg \cap [\lin h = 0]\right),\]
where $[h_P=0]$ denotes the polyhedral subcomplex of $\fan_P$ consisting of those polyhedra on which $h_P$ vanishes.
\end{enumerate}

By \cite{tidiv}, every invariant Cartier divisor arises in this way. 
We have $D_h \sim D_{h'}$ if and only if $h_P - h'_P$ is affine linear for every $P$, i.e. $h_P - h'_P = \langle u, \cdot \rangle + a_P$, and $\sum_P a_P = 0$. 
Moreover, we have a criterion for ampleness expressed in the following notation. 
We denote by $\Box_h$ the polytope given by
\begin{equation}
  \label{eq:1}
  \Box_h = \{u \in M_\QQ \mid \langle u, \cdot \rangle \geq (\lin h)\}
\end{equation}
and consider concave piecewise affine function $h^*_P$ on $\Box_h$ as ``dual'' of $h_P$:
\[h_P^*(u) := \inf_v (\langle u, v \rangle - h_P(v)).\]
The definition implies that $h_P^*(u)$ is finite for $u \in \Box_h$.

\begin{theorem}
\label{thm:ample}
  If $D_h$ is ample, then $h_P$ is strongly concave for every $P \in \PP^1$ and $h_P^*(u) \geq 0$ for every $u \in \Box_h$.
\end{theorem}
\begin{proof}
  By \cite[Theorem~3.28]{tidiv}, $h_P$ has to be strongly concave and by \cite[Prop.~3.1(i)]{IS10} we get that $h_P^*(u) \geq 0$.
\end{proof}

\section{Cones over projective T-varieties}\label{sec:tvars}
From now one we assume that the $T$-varieties we consider are proper over the base, i.e.~that all slices $\fan_P$ are \emph{subdivisions} of $N_\QQ$.

\begin{definition}[Equivariant covering by toric charts]
  A $T$-variety is called \emph{equivariantly covered by toric charts}, 
  if there is an open covering by toric varieties $U_i$ such that the torus $T$ acts as a subtorus of the embedded torus of $U_i$.
\end{definition}

\begin{lemma}
\label{lem:tcovered}The T-variety $X(\fan)$ is equivariantly covered by toric charts if and only if for every maximal polyhedron $\Delta$ in $\fan_P$, $P\in\PP$, all but at most two slices contain a lattice translate of $\tail(\Delta)$. 
  In particular, either $X(\fan)$ itself is toric or there is at most one $P \in \PP^1$ such that $\fan_P$ does not contain a lattice vertex. 
\end{lemma}
\begin{proof}
  The first part is a corollary of Lemma~\ref{lem:toric}. 
  To prove the last statement, we consider two points $P,Q$ such that $\fan_P,\fan_Q$ contain only non-lattice vertices. 
  Now, consider a third point $R$ and a maximal polyhedron $\Delta \subset \fan_R$. 
  Since there is no lattice translate of $\tail(\Delta)$ in $\fan_P$ and $\fan_Q$, $\Delta$ itself must be a translated cone. 
  Notice that all maximal cones in $\fan_R$ must be translated by the same lattice point. 
  Indeed, otherwise they would not cover $N_\QQ$ and there would exist a different maximal dimensional polyhedron
   that could not be a lattice translate of its tail cone. 
   Hence, for any $R\notin\{P,Q\}$ the slice $\fan_R$ is just a translated tail fan. By Lemma~\ref{lem:toric}, $X(\fan)$ is a toric variety.
\end{proof}

\begin{remark}
  By \cite[Appendix]{APS13} this criterion is fulfilled for all smooth rational $T$-varieties of complexity one. Hence, they are covered by affine spaces.
\end{remark}

To get flexibility for every affine cone we need to strengthen the condition in Lemma~\ref{lem:tcovered}.
\begin{theorem}
\label{thm:all-flexible}
   Let $X=\X(\fan)$ be a T-variety such that for any maximal polyhedron $\Delta \in \fan_y$, $y\in\PP^1$, 
   at most two slices contain a polyhedron with the same tail cone $\tail(\Delta)$ that is not a lattice translate of $\tail(\Delta)$. 
   Then for every very ample divisor $H$ the corresponding  affine cone is flexible.
\end{theorem}
\begin{proof}
By Theorem~\ref{thm:flexible-criterion} it is enough to show that there exists an equivariant $H$-polar covering by toric charts.
 Let us first rephrase this condition in terms of f-divisors.
  For every maximal polyhedron $\Delta_P \in \fan_P$ there exists a strictly concave non-positive support function $h$ on $\fan$ corresponding to an effective divisor $D_h\sim H$, 
  such that $[h_P=0]=\Delta_P$ (i.e. $h_P|_{\Delta_P}\equiv 0$  and negative elsewhere).
   For being a toric covering additionally we have to impose that $[h = 0]$ has only two slices that are not lattice translates of the tail fan.

We now construct such a covering for some very ample divisor $H=D_h$ corresponding to a support function $h$. 
Fix a maximal polyhedron $\Delta \subset \fan_Q$.
Then $h_Q|_{\Delta}$ is affine linear, i.e. $h_Q(v)=\langle u, v \rangle + a$. 
By concavity this implies $u \in \Box_h$, with $\Box_h$ defined as in (\ref{eq:1}). 
We now consider $h':=h-u$ with $h_P'(v):=h_P(v) - \langle u, v \rangle$. 
Now, $h'_P$ is again strongly concave and obtains its maximum at a polyhedron $\Delta_P$ with tail cone $\tail(\Delta_P) = \tail(\Delta)$. 
Moreover, by construction we have $0 \in \Box_{h'}$.

By our precondition, we may assume without loss of generality that for every point $R\in\PP\setminus\{0,\infty\}$
the polyhedron $\Delta_R$ is a lattice translate of $\tail( \Delta)$. Assume further $Q\neq\infty$
and introduce $h^\infty$ by
\[
\begin{cases}
  h_P^{\infty}(v) :=h'_P(v) - \max\mathrm{Im\,} h'_P\quad \mbox{ for } P \neq \infty,\\
 h_\infty^\infty(v) := h'_{\infty}(v) + \sum_{P \neq \infty} \max\mathrm{Im\,}  h'_P. 
\end{cases}
\]
It remains to check that $h'_\infty(v) + \sum_{P \neq \infty} \max \mathrm{Im\,} h^\prime_P \leq 0$ to see that $D_{h^\infty}$ is indeed effective.
Recall that we have $0 \in \Box_{h'}$. The claim follows from the ampleness of $D_{h'}$ and Theorem~\ref{thm:ample}.

Now, by construction we have 
$D_{h^\infty}\sim H$ and $[h^\infty_Q=0]=\Delta$. Moreover, $[h_P^\infty=0]$ is a lattice translate of $\tail(\Delta)$ for each $P\notin\{0,\infty\}$. Then it describes a toric chart.



Taking these toric charts for every maximal polyhedron provides us with an $H$-polar covering. 
Now, our result follows by Theorem~\ref{thm:flexible-criterion}.
\end{proof}

\conefanos
\begin{proof}
 For all Fano threefolds from Theorem \ref{thm:cones-over-Fanos} the corresponding f-divisors are listed in \cite{suess:picbook}. 
 One can easily check that the precondition of Theorem~\ref{thm:all-flexible} is fulfilled in every case.
\end{proof}
\begin{example}
Let us illustrate the difference of assumptions in Lemma \ref{lem:tcovered} and Theorem \ref{thm:all-flexible}. 
In the lemma we are allowed to \emph{choose} the polyhedron with the given tail cone. 
Hence, if we consider the variety given by the slice $(-\infty,-1],[-1,1],[1,\infty)$ taken three times, then it does satisfy the assumptions. 
Indeed, if we take the maximal polytope $[-1,1]$ in one slice, in other two slices we can take just the vertex $\{1\}$, which is a lattice shift of the tail cone $\{0\}$. 
On the other hand in the theorem we ask for \emph{all} polyhedra with the given tail cone. 
Here, we get three times $[-1,1]$ which is not a lattice translate of $\{0\}$. 
Such a difference is only possible for cones that are not full-dimensional.
\end{example}
\begin{example}
  We are coming back to the blowup of the quadric threefold from Example~\ref{exp:fdiv}. 
  We may check that the corresponding f-divisor in Figure~\ref{fig:fdiv} fulfills the condition of Theorem~\ref{thm:all-flexible}. 
  Hence, all affine cones over the blowup of the quadric threefold are flexible.
\end{example}

\begin{example}
  The hypersurface
  $\VV(x_0y_0^2+x_1y_1^2+x_2y_2^2) \subset \PP^2 \times \PP^2$
  is 2.24 from our list in Theorem~\ref{thm:cones-over-Fanos}. 
  Hence, every affine cone over this variety is flexible. In particular, this is true for the cone over the Segre embedding.
\end{example}

\section{Total coordinate spaces}
\label{sec:total-coord-spaces}

We recall the definition of Cox rings.
\begin{definition}[Cox sheaf, Cox ring, universal torsor, total coordinate space]
  Let $X$ be a normal variety, whose class group is a free abelian group generated by divisor classes $D_1,\ldots,D_r$. The \emph{Cox sheaf} of $X$ is defined by
\[\Cox = \bigoplus_{\vec{a} \in \ZZ_{\geq 0}^r} \CO\left({\textstyle \sum_{i=1}^r}\; a_i D_i\right).\]
It becomes a sheaf of $\CO_X$-algebras via the usual multiplication of sections. Its global sections are called the \emph{Cox ring} denoted by $\Cox(X)$.

The relative spectrum $\hat X = \Spec_X(\Cox)$ is called the \emph{universal torsor} of $X$. 
It is an open subset of the absolute spectrum $\overline X = \Spec(\Cox(X))$, which is called the \emph{total coordinate space} of $X$.
 By construction, the Cox ring is graded by the class group of $X$ inducing an action of the torus $\Spec k[\Cl(X)]$ on the total coordinate space.
\end{definition}

In the following we are studying flexibility of total coordinate spaces for several classes of varieties.

\subsection{Del Pezzo surfaces}
\label{sec:del-pezzo-surfaces}
Since the smooth del Pezzo surfaces of degrees 6,7,8, and 9 are toric, their total coordinate spaces are just affine spaces and hence flexible. 
The remaining del Pezzo surfaces are blowups $X_r$ of $\PP^2$ in $r$ points of general position, where $4 \leq r \leq 8$. 
Their Cox rings are described for example in \cite{zbMATH02158894}, \cite{zbMATH05168539}, and \cite{zbMATH05652990}.

An exceptional curve on $X$ is a curve of self-intersection $-1$ and anti-canonical degree $1$. 
On every del Pezzo surface there are only finitely many of them. 
We will  use  the following facts from \cite{zbMATH02158894}.
\begin{theorem}[{\cite[Thm 3.2 \& Prop. 3.4]{zbMATH02158894}}]
\label{thm:delPezzo-cox-generators}
Let $e_1 \ldots, e_N$ be sections corresponding to exceptional curves on a del Pezzo surface $X_r$ and $I$ be the ideal of their relations. Then
\begin{enumerate}
\item $\Cox(X_r)=\K[e_1,\ldots,e_N]/I$ for $4\leq r\leq 7$; \label{item:generators1}
\item $\Cox(X_8)=\K[e_1,\ldots,e_N]/I\oplus\langle f_1,f_2\rangle_\K$ as a vector space, \\
where 
$f_1, f_2 \in H^0(X_8,\CO(-K_{X_8}))$ are elements of degree one with respect to
 the $\ZZ$-grading by the anti-canonical degree of a divisor class.
\label{item:degree2}
\end{enumerate}
\end{theorem}

Theorem~\ref{thm:delPezzo-cox-generators} shows that the Cox ring is generated by elements of degree $1$ and $Y_r:=\Proj(\Cox(X_r))$ comes with an embedding into $\PP^{N-1}$, 
where $N$ (resp. $N-2$) is the number of exceptional curves in the case $4\leq r \leq 7$ (resp. $r=8$).

 In this situation the total coordinate space $\overline X_r$ is the affine cone over this embedding.

\begin{proposition}
\label{prop:blowup-coxring}
Let $e$ be a section corresponding to an exceptional curve. Then the principal open subset $(Y_r)_e$ is isomorphic to $X_{r-1}$.
\end{proposition}
\begin{proof}
  This can be found for example in the proof of Proposition 3.4 in \cite{zbMATH02158894}.
\end{proof}

\coxdPflex
\begin{proof}
  As said above, it is enough to check $X_r$ with $4 \leq r \leq 8$. We will go by induction. 
  The del Pezzo surface $X_3$ is toric. Therefore it has a flexible total coordinate space. 
  Now, consider $X_r$ with  $4 \leq r \leq 8$. Then we have seen that $\overline X_r$ is the affine cone over $Y_r$. 
  Moreover, principal open subsets corresponding to sections of exceptional curves are isomorphic to $\overline X_{r-1}$ and hence flexible by induction hypothesis.
   It remains to check that these principal open subsets cover $Y_r$ to conclude flexibility of $\overline X_r$ from Theorem~\ref{thm:flexible-criterion}. 
   For $4 \leq r \leq 7$ this follows directly from Theorem~\ref{thm:delPezzo-cox-generators}~(\ref{item:generators1}). 
   For the case $r=8$ we have to take care for the remaining generators. 
   By Theorem~\ref{thm:delPezzo-cox-generators}~(\ref{item:degree2}) their squares are contained 
in the ideal $(e_1 \ldots, e_N)$ generated by the sections corresponding to exceptional curves, 
but then the common vanishing of $e_1 \ldots, e_N$ implies the vanishing of the remaining generators and hence $Y_r = \bigcup_{i=1}^N (Y_r)_{e_i}$.
\end{proof}

\subsection{Smooth complexity-one T-varieties}
In \cite{cox} the Cox rings of T-varieties are studied. For the case of complexity-one action they have a very particular form.
\begin{proposition}[{\cite[Corollary 4.9]{cox}}]
\label{prop:cox}
  Let $\fan$ be an f-divisor and let us denote  by $\fan^\times$ the subset of rays in $\tail(\fan)^{(1)}$ that do not intersect $\sdeg$. 
  Then the Cox ring of $X(\fan)$ is given by
  \[\frac{\KK[S_\rho, T_{P,v} \mid \rho \in \fan^\times,  P \in \supp \fan, v \in \fan_P^{(0)}]}
{\langle z \cdot T^{\mu(0)} + T^{\mu(\infty)} + T^{\mu(z)} \mid z \in \supp \fan \cap \KK^*\rangle}\]
where $T^{\mu(P)} := \prod_{v \in \fan_P^{(0)}} T_{P,v}^{\mu(v)}$ and $\mu(v)$ denotes the minimal positive integer such that $\mu(v) \cdot v$ is a lattice element.
\end{proposition}

If we impose the additional condition that the T-variety is equivariantly covered by toric charts 
(which is fulfilled in the smooth case), then we can conclude the following.
\begin{proposition}\label{T1-trinom}
  The Cox ring of a complexity-one T-variety equivariantly covered by toric
  charts is isomorphic to
\begin{equation*}
\KK[S_1,\ldots,S_{n_S};T_{\ell,j}\mid 0\le \ell\le m, 1\le j\le n_\ell]/\langle {z_\ell \cdot A_0 +  A_1 + A_\ell}  \mid 2 \leq \ell \leq m \rangle,
\end{equation*}
where
\begin{enumerate}
\item $N, m, n_0,\ldots,n_m\in\ZZ_{>0}$;
\item $z_2,\ldots, z_m$ are distinct elements of $\KK^*$;
\item  for $\ell=0,\ldots,m$, $A_\ell$ is a monomial in $\KK[T_{\ell,1},\ldots,T_{\ell,n_\ell}]$;
\item for $\ell=1,\ldots,m$ the monomial $A_\ell$ is linear in at least one variable.\label{T1-trinom-lin}
\end{enumerate}
Moreover, if $X$ is Fano, then we may assume that $A_\ell$ for $\ell \geq 3$ is linear in each variable.
\end{proposition}
\begin{proof}
  The first statements follow directly from Proposition~\ref{prop:cox} and Lemma~\ref{lem:tcovered}. 
  For the Fano case note that the Cox ring of a Fano variety is log-terminal by \cite{brown2011singularities,gongyo2012characterization} and factorial by \cite{1073.14001}. 
  Hence, by Remark 6.4 in \cite{tsing} we obtain the last statement.
\end{proof}

\begin{proposition}\label{T-Cox-eq-flex}
Let $X$ be a complexity-one T-variety equivariantly covered by toric charts
and $\overline{X}$ be a total coordinate space of the Cox ring of $X$.
Then $\overline{X}$ is flexible.
\end{proposition}
\begin{proof}
Let $\K[\overline{X}]$ be as in Proposition~\ref{T1-trinom}. 
We assume that $\overline{X}$ is naturally embedded into an affine space $\A^N=\Spec\K[S_1,\ldots,S_N;T_{\ell,j}\mid 0\le \ell\le m, 1\le j\le n_\ell].$

The images of monomials $A_0,\ldots,A_m$ in $\K[\overline{X}]$ 
span a two-dimensional subspace, and no two of them are
collinear. Therefore, we may permute $A_0,\ldots,A_m$ along with a proper
change of their coefficients, indices of variables, and numbers $z_i$.

\begin{lemma}
The point $x\in \overline{X}$ is singular if and only if there are at least three monomials $A_i$ such that all their partial derivatives are vanishing at $x$.
\end{lemma}
\begin{proof}
Let $x$ be singular and denote $L_\ell=z_\ell \cdot A_0 +  A_1 + A_\ell$ for $\ell=2,\ldots,m,$
then there is a non-trivial linear combination $L\in\langle L_2,\ldots,L_m\rangle_\K$, whose partial derivatives vanish at $x$. 
Since $L$ is a sum of at least three monomials $A_i$, whose  partial derivatives also vanish, the statement follows.

Conversely, given three monomials with partial derivatives vanishing at $x$, we assume that they are $A_0,A_1,A_2$.
Then we take $L=L_2$.
\end{proof}

So, for a smooth point $x\in X$, up to permutation of monomials and variables we assume that each monomial $A_i,\,i=2,\ldots,m,$
 has a non-zero partial derivative, say, $\frac{\partial A_i}{\partial T_{i, 1}}(x)\neq 0$.
 Moreover, we may choose $T_{i,1}$  to be linear in $A_i$. Indeed, if $T_{i,1}$ is not linear, then $A_i(x)\neq0$,
and we take $T_{i,1}$ to be a linear variable by \ref{T1-trinom}.(\ref{T1-trinom-lin}).
We denote $B_i=\frac{\partial A_i}{\partial T_{i,1}}$, which is non-zero at $x$.

Given a set of arbitrary numbers $c_{0,1},\ldots,c_{0,n_0},c_{1,1},\ldots, c_{1,n_1}\in\KK,$ we  construct a $\G_a$-action $\phi$ on $\A^N$ in two steps. 
First, denoting a parameter of $\G_a$ by $t$, we let
\begin{align*}
\phi^*\colon \KK[\A^N]&\to\KK[\A^N]\otimes\KK[t],&\\
  S_i &\mapsto S_i, &\mbox{ for } i=1,\ldots,n_S,\\
 T_{0,j}&\mapsto T_{0,j}+t c_{0,j}\prod_{k=2}^m B_k,&\mbox{ for } j=1,\ldots,n_0,\\
 T_{1,j}&\mapsto T_{1,j}+t c_{1,j}\prod_{k=2}^m B_k,&\mbox{ for } j=1,\ldots,n_1.\\
\end{align*}
Then, for some $H_0,H_1\in\K[\A^N]\otimes\KK[\G_a]$,
\begin{align*}
\phi^*(A_0)=&A_0+ H_0\prod_{k=2}^m B_k,\\
\phi^*(A_1)=&A_1+ H_1\prod_{k=2}^m B_k.
\end{align*}
Now, for each $\ell=2,\ldots,m$ we let
\begin{align*}
 T_{\ell,1}&\mapsto T_{\ell,1}-(z_\ell\cdot H_0+H_1)\prod_{\substack{2 \leq k\leq m\\ k\neq \ell}} B_k,\\
 T_{\ell,j}&\mapsto T_{\ell,j},\qquad\mbox{ for } j=2,\ldots,n_\ell.\\
\end{align*}
Then the trinomial $z_\ell\cdot A_0+A_1+A_\ell$ is fixed by $\phi^*$, so $\phi$ preserves $\overline{X}\subset\A^N.$ Thus, we have constructed a $\G_a$-action on the total coordinate space $\overline{X}$, which we also denote by $\phi$.

As said before, for a chosen smooth point $x$ we have $B_i(x)\neq 0$.
Let us take another smooth point $y\in\overline{X}$ with non-zero coordinates and
move $x$ to $y$ by $\G_a$-actions, denoting images of $x$ by same letter.
By translations along coordinates $S_1,\ldots,S_{n_S}$ we `equalize' them, i.e., obtain $S_i(x)=S_i(y)$, $i=1,\ldots,n_S.$ 

Since $c_{0,1},\ldots,c_{0,n_0},c_{1,1},\ldots, c_{1,n_1}\in\KK$ are arbitrary, with $\phi$
we also equalize coordinates $T_{0,j},\,j=1,\ldots,n_0,$ and $T_{1,j},\,j=1,\ldots,n_1,$ at $x$ and $y$.
Now, let $T_{1,1}$ be a linear variable in $A_1$, then for each $\ell=2,\ldots,m$ we construct 
a $\G_a$-action $\phi_\ell$ by permuting monomials $A_1$ and $A_\ell$ and applying the procedure above. 
Since $\frac{\partial A_1}{\partial T_{1,1}}(x)=\frac{\partial A_1}{\partial T_{1,1}}(y)\neq0$,  
with $\phi_\ell$ we may equalize coordinates $T_{\ell,2},\ldots,T_{\ell,m}$, 
but break the equality of coordinate $T_{1,1}$, which we restore with $\phi$. 

Proceeding this way for each $\ell=2,\ldots,m$, we equalize all coordinates except  $T_{2,1},\ldots,T_{m,1}$.
But in this case the equation $z_\ell\cdot A_0+A_1+A_\ell=0$ with condition $B_\ell(x),B_\ell(y)\neq0$ implies $T_{\ell,1}(x)=T_{\ell,1}(y)$ for each $\ell$.
So, we may send any smooth to any point with non-zero coordinates, hence
 the action of $\SAut \overline{X}$ is transitive on smooth points of $\overline{X}$.
\end{proof}
\coxTflex
\begin{proof}
By \cite[Thm~A.1]{APS13}, every smooth complete rational T-variety of complexity one is covered by affine spaces, so the statement follows from Proposition~\ref{T-Cox-eq-flex}.
\end{proof}

\subsection{Flexibility of total coordinate spaces vs. flexible coverings}
In \cite{APS13} it was proved that for a variety with an open covering by affine spaces one obtains flexibility of the universal torsor. 
However, it is not clear whether the flexibility property extends to the total coordinate space. 
This motivates the following even more general question.
\begin{question}
  Provided a variety admits a covering by flexible affine subsets, does this imply flexibility of the total coordinate space?
\end{question}

It is also very tempting to try to connect flexibility of the total coordinate space of the Cox ring of $X$ with that of affine cones over $X$.  
The following example shows that flexibility of the total coordinate space does not imply flexibility of all affine cones.

\begin{example}[Del Pezzo surfaces]
  We have seen in Section~\ref{sec:del-pezzo-surfaces} that all total coordinate spaces of del Pezzo surfaces are flexible. 
  On the other hand, del Pezzo surfaces are covered by affine spaces, which are flexible. 
  Concerning flexibilty of affine cones it was shown in \cite{flexcones, zbMATH06551221} that for degree 4 and 5 all affine cones are flexible, 
  but by \cite{zbMATH06290382,cheltsov16} the anti-canonical cones over del Pezzo surfaces of degree 3, 2, and 1 are not flexible. 
\end{example}


One may still ask if flexibility of all affine cones implies flexibility of the total coordinate space or for a more subtle relation, e.g.~involving the grading of the Cox ring. 
\begin{question}
  Is there a relation between flexibility of the total coordinate space of $X$ and the fact that \emph{all} affine cones over $X$ are flexible?
\end{question}

Let us give some illustrating examples for these questions.

\begin{example}[Toric varieties]
  The Cox ring of a toric variety is polynomial. Hence, the total coordinate space is flexible. On the other hand,
  the torus invariant affine charts and also the affine cones of a toric variety are again toric and hence flexible by \cite[Theorem 0.2]{flex-instances}.
\end{example}

\begin{example}[Blowups of $\PP^n$ in cubic hypersurfaces inside hyperplanes]
  The blowup constructions from Example~\ref{exp:blowup} give varieties for which all affine cones are flexible, as we have seen. 
  On the other hand, the total coordinate space is flexible, as wee see in the following.

We can consider the $k^*$-action on $\PP^n$ given by multiplication with the coordinate $x_0$.  
It comes with a natural quotient map to $\PP^{n-1}$ being defined outside the isolated fixed point $(1\mathbin:0:\ldots:0)$. 
Then the centers of our blowups are fixed points of the action and we obtain induced actions on $X$ and $X'$ with natural quotient maps given by composition of the original quotient map with the blowup. 
Now, we may use Theorem~1.2 in \cite{cox} to calculate the Cox rings
 \[\Cox(X) \cong \K[T_0,\ldots,T_{n},T_n',S_1]/(T_nT_n'-f(T_0,\ldots,T_{n-1},0))\] and 
\[\Cox(X') \cong \K[T_0,\ldots,T_{n},T_n',S_1,S_2]/(T_nT_n'-f(T_0,\ldots,T_{n-1},0))).\]
We see that they are suspensions over an affine space and hence flexible by Theorem~0.2 in \cite{flex-instances}.
\end{example}



\begin{example}[T-varieties of complexity one] 
  The proof of Theorem~\ref{T1-flex} implies that for T-varieties of complexity one the condition of being covered by toric (and hence flexible) charts is enough to obtain flexibility of the total coordinate space. 
  On the other hand, to conclude flexibility of all affine cones we had to impose the stronger (technical) condition of  Theorem \ref{thm:all-flexible}.
\end{example}

\bibliographystyle{alpha}
\bibliography{tflex}

\end{document}